\documentclass[12 pt,a4paper]{amsart} 
\usepackage{amsfonts,amssymb,amscd,amsmath,enumerate,verbatim,calc} 
\usepackage{amsthm}
\newtheorem{theorem}{Theorem}[section]
\newtheorem{definition}{Definition}[section]

\newtheorem{proposition}[theorem]{Proposition}
\newtheorem{remark}[theorem]{Remark}
\usepackage{tikz-cd}
\usepackage{mathtools}

\newtheorem*{example}{Example}
\numberwithin{equation}{section}

\usepackage{amsfonts}

\renewcommand{\c}{\circ}

\renewcommand{\ker}{\textnormal{ker}}

\newcommand{\gyr}{\textnormal{gyr}}
\newcommand{\orb}{\textnormal{Orb}}
\newcommand{\stb}{\textnormal{stab}}
\newcommand{\sym}{\textnormal{Sym}}

\begin{document}
	\title{Gyrogroup through its Grothendieck Group Completion and Right gyrogroup action}

\author{Akshay Kumar$^{1}$,
Mani Shankar Pandey$^{2}$,  Seema Kushwaha$^{3}$ AND \\Sumit Kumar Upadhyay$^{4}$\vspace{.4cm}\\
{Department of Applied Sciences,\\ Indian Institute of Information Technology Allahabad\\Prayagraj, U. P., India} }

\thanks{$^1$mehtaaksh4@gmail.com, $^2$manishankarpandey4@gmail.com, $^3$seema28k@gmail.com, $^4$upadhyaysumit365@gmail.com}
	
%
	
		
	\begin{abstract} In this article, we discuss the Grothendieck group completion (GGC) of a gyrogroup. Consequently, we show that there is a one to one correspondence between actions and representations of a gyrogroup, and actions and representations of its Grothendieck group completion. We also introduce the concept of an action of a right gyrogroup.
	\end{abstract}
	
\maketitle
\section{Introduction}
In 1988,  Ungar introduced a non-associative algebraic structure, {\it gyrogroup,} which is a generalization of groups. 
The concept of gyrogroups first came into the picture
in the study of Einstein’s relativistic velocity addition law \cite{ungar_thomas,ungar_book}. The focus of this article is around the  following problem. 

For a given gyrogroup $G$, does there exists a pair $(M(G),\nu), $ where $M(G)$ is a group and $\nu$ is a gyrogroup homomorphism  from $G$ to $M(G)$ such that for any given group $K$ and a gyrogroup homomorphism $\eta$ from $G$ to $K,$ there exists a unique group homomorphism $\alpha$ from $M(G)$ to $K$ such that  $\alpha\circ\nu=\eta?$

We answer  this question in Section 2 affirmatively by using the idea of Grothendieck which is used in \cite{rjl_weak_classification} and we call it here the Grothendieck group completion (GGC) of a gyrogroup.

In recent articles \cite{suksumran_lagrange,suksumran_isomorphism}, it has been shown that gyrogroups share significant analogy with groups. It is shown that most of the traditional and remarkable results of group theory like Lagrange theorem, the fundamental isomorphism theorem, Cayley theorem are still true for gyrogroups. 
In the continuation of this,  Suksumran \cite{suksumran_gyroaction} explored the notion of gyrogroup actions which turns out to be  a natural generalization of the usual notion of group action. Later, the author \cite{suksumaran_maschke,suksumaran_reducibility,suksumran_repre} generalized the notion of representation of a group for gyrogroups. The author achieved many well known results of group actions in the theory of gyrogroup action and representation. 
He defines an action $*$ of a gyrogroup $(G, \circ)$ on a set $X$ by the requirement that $(a\circ b)*x = a* (b*x) $ and a linear representation $\rho: G \longrightarrow GL(V)$ by putting $\rho(a\circ b) = \rho(a)\c \rho(b)$, for all $a, b \in G$.

In this article, we show that this is equivalent to having an action of the group $M(G)$ on $X$ and a linear representation of $M(G)$ on $V$, where $M(G)$ is the Grothendieck group completion of $G$. Indeed, the category of gyrogroup actions is equivalent to the category of group actions and the category of gyrogroup representations is equivalent to the category of group representations. As such, most of the results are deducible from the known results for groups by exploring the quotient map $\nu : G \longrightarrow M(G)$. We illustrate this in Section 3 and 4 of the paper. 

However, we introduce a more realistic notion of right gyrogroup action/representation and try to develop the theory accordingly. The Section 5 is devoted  to the right gyrogroup action.

Now, we recall the definition of a gyrogroup and its properties which we will use further.

\begin{definition} \label{gyrogroup}  A pair $(G,\c)$ consisting of a non-empty set $G$ and a binary operation $\c$ on $G$ is called a {\it gyrogroup} if its binary operation $\c$ satisfies the following axioms. 
	\begin{enumerate}
		\item There exists an element $e\in G$, called
		a left identity, such that $$e\c a=a\,\,\text{ for\, all } a\in G.$$
\item For each $a\in G,$ there exists an element $a^{-1}\in G$, called a left inverse of $a,$ such that $a^{-1}\c a=e$.
\item For $a,b\in G,$ there exists an automorphism $\gyr[a,b]$ such that 
$$ a\c(b\c c)= (a\c b)\c \gyr[a,b]c$$
for all $a,b,c\in G.$
\item $\gyr[a\c b, b]=\gyr[a,b]$ for all $a,b,c\in G.$
	\end{enumerate} 
	\end{definition}

	\begin{proposition}\cite[Theorem 2.34]{ungar_book}\label{gyro identity}\label{gyro identity_1}
	Let $(G, \c)$ be a gyrogroup. Then
 $\gyr^{-1}[a, b]=\gyr[b,a]$ for all $a,b\in G.$
	\end{proposition}
\begin{proposition}\cite[Theorem 2.13]{ungar_book}\label{gyro identity}\label{gyro identity_2}
	Let $(G, \c)$ be a gyrogroup. Then
	$$\gyr[a\c b,\gyr[a,b](c)]\gyr[a,b]=\gyr[a,b\c c]\gyr[b,c]$$ for all $a,b, c\in G.$
\end{proposition}

\begin{proposition}\cite[Proposition 6]{suksumran_isomorphism} \label{l-subgyro} Let $G$ be a gyrogroup and let $X \subseteq G$. Then the following are equivalent\\
1) $\gyr[a, b](X) \subseteq X$ for all $a, b \in G$,\\
2) $\gyr[a, b](X) = X$ for all $a, b \in G$.
\end{proposition}

In this article, most of the terminologies and definitions related to gyrogroup actions are adopted from \cite{suksumran_gyroaction}.

	\section{Grothendieck group completion of a gyrogroup}
	In this section, we construct the Grothendieck group completion (GGC) of a gyrogroup. 
	
Let $(G,\c)$ be a gyrogroup. For $a,b\in G$, the map $\gyr[a,b]$ given by 
$$\gyr[a,b]z=(a\c b)^{-1}\c (a\c(b\c z))$$
belongs to $\sym(G)$. We know that $\gyr[a,b]e=e$, therefore, $\gyr[a,b]\in \sym(G\setminus\{e\})$. 
\begin{proposition}\label{identities}
	Let $f\in \sym(G\setminus\{e\})$ and $a\in G$. Define $\sigma_a(f):G\setminus\{e\}\to G\setminus\{e\}$ by the equation
	$f(a)\c\sigma_a(f)(b)=f(a\c b),$ for all $b\in G\setminus\{e\}$. Then
	\begin{enumerate}
		\item $\sigma_a(f)\in \sym(G\setminus\{e\})$ for all $a\in G$ and $f\in \sym (G\setminus\{e\})$;
		\item $\sigma_a(I)=I,$ where $I$ is the identity map on $G\setminus\{e\}$;
		\item $(f^{-1}(a))^{-1}=\sigma_{a}(f^{-1})(a^{-1})$;
		\item $\sigma_a(\gyr[b,c])=\gyr[b,c],$ for every $a\in G$;
		\item $\sigma_a(f_1f_2)=\sigma_{f_2(a)}(f_1)\sigma_a(f_2)$;
		\item $(\sigma_a(f))^{-1}=\sigma_{f(a)}(f^{-1})$;
		
		\item $\gyr[f(a), \sigma_{a}(f)(a^{-1})]=I$;
		\item $\sigma_{(a\c b)}(f)\gyr[a,b]=\gyr[f(a),\sigma_a(f)(b)]\sigma_b(\sigma_a(f)).$
	\end{enumerate}
	\begin{proof}
		Statement $(1)$, $(2), (3)$ and $(4)$ are clear from the definition. Now we will give the proof of identity (5).
		
		Let $f_1,f_2\in \sym(G\setminus\{e\})$ and $a\in G.$ For $z\in G,$
		\begin{align*}	
			f_1f_2(a)\c 	\sigma_a(f_1f_2)(z)&= f_1f_2(a\c z)\\
			f_1f_2(a)\c 	\sigma_a(f_1f_2)(z) & =f_1(f_2(a\c z))\\
			f_1f_2(a)\c 	\sigma_a(f_1f_2)(z)&=f_1(f_2(a)\c \sigma_{a}(f_2)(z))\\
			f_1f_2(a)\c 	\sigma_a(f_1f_2)(z)&=f_1f_2(a)\c \sigma_{f_2(a)}(f_1)(\sigma_a(f_2)(z))\\
			\sigma_a(f_1f_2)(z)&= \sigma_{f_2(a)}(f_1)\sigma_a(f_2)(z)
		\end{align*}
Therefore, $\sigma_a(f_1f_2)=\sigma_{f_2(a)}(f_1)\sigma_a(f_2).$ Statement $(6)$ follows from $(2)$ and $(5)$ by taking $f_2=f$ and $f_1=f^{-1}.$
		Statement (7) follows from (3). Now we will give the proof of identity (8).
		\begin{align*}f(a\c b)\c \sigma_{a\c b}(f)(z)&=f((a\c b)\c z))\\
			&=f(a\c (b\c \gyr[b,a]z))\\
			&= f(a)\c \sigma_a(f)(b\c \gyr[b,a](z))\\
			&= f(a)\c (\sigma_a(f)(b)\c \sigma_b(\sigma_a(f))(\gyr[b,a](z)))\\
			&=(f(a)\c \sigma_a(f)(b))\c \gyr[f(a), \sigma_a(f)(b)](\sigma_b[\sigma_a(f)]\gyr[b,a](z))\\
			&=f(a\c b)\c \gyr[f(a), \sigma_a(f)(b)](\sigma_b[\sigma_a(f)]\gyr[b,a](z))\\
			\sigma_{a\c b}(f)z&=\gyr[f(a), \sigma_a(f)(b)](\sigma_b[\sigma_a(f)]\gyr[b,a](z)) \textnormal{(by left cancellation)}\\
	\Rightarrow		\sigma_{a\c b}(f) \gyr[a,b]&=\gyr[f(a), \sigma_a(f)(b)](\sigma_b[\sigma_a(f)]).
		\end{align*}

	\end{proof}
\end{proposition}
\noindent
Set $G^*= G\setminus\{e\}$. Let $G\times \sym(G^*)$ denote the Cartesian product of $G$ and $\sym(G^*)$. We define a binary operation $\cdot$ on $G\times \sym(G^*)$ given by
$$(a,f)\cdot (b,g)=(a\c f(b), \gyr[a,f(b)]\sigma_b(f)g).$$

\begin{theorem} The groupoid $(G\times \sym(G^*),\,\cdot\,)$ is a group.
\end{theorem}
\begin{proof}
\textbf{Associativity:} 	Let $(a,f), (b,g), (c,h)\in G\times \sym(G^*).$
	Then
	\begin{align*}
		(a,f)\cdot((b,g)\cdot(c,h))&=\left(a,f)\cdot(b\c g(c), \gyr[b,g(c)]\sigma_c(g)h\right)\\
		&= \left(a\c f(b\c g(c)), \gyr[a,f(b\c g(c))] \sigma_{b\c g(c)}(f)\gyr[b,g(c)]\sigma_c(g)h\right )
	\end{align*}
	and 
	\begin{align*}
		\left(	(a,f)\cdot((b,g)\right)\cdot(c,h))&=\left(a\c f(b),  \,\gyr[a,f(b)]\sigma_b(f)g\right)\c (c,h)\\
		&= \left((a\c f(b))\c \gyr[a,f(b)]\sigma_b(f)g(c),\, \tau )\right )\\
		&= (a\c (f(b)\c  \sigma_b(f)g (c),\, \tau )\\
		&= (a\c f(b\c g(c)),\,\tau )
	\end{align*}
	where $\tau=\gyr\left[a\c f(b), \gyr[a,f(b)]\sigma_b(f)g(c)\right]\sigma_c(\gyr[a,f(b)]\,\sigma_b(f)g)h.$
	Note that 
	\begin{align*}
		\sigma_c(\gyr[a,f(b)]\sigma_b(f)g)&=\gyr[a,f(b)]\,\sigma_c(\sigma_b(f)g)~~~~~~~~\textnormal{by Proposition \ref{identities}(5)}\\
		&=\gyr[a,f(b)]\,\sigma_{g(c)}(\sigma_b(f))\sigma_c(g)
	\end{align*}
	By Proposition \ref{identities}(8), $\sigma_{g(c)}(\sigma_b(f))=\gyr[\sigma_b(f)(g(c)), f(b)] \sigma_{b\c g(c)}(f)\gyr[b,g(c)]$ so that
		\begin{equation*}
		\sigma_c(\gyr[a,f(b)]\sigma_b(f)g)=\gyr[a,f(b)]\,  \gyr[\sigma_b(f)(g(c)), f(b)] \sigma_{b\c g(c)}(f)\gyr[b,g(c)]   \sigma_c(g).
	\end{equation*}
It gives 
\begin{align*}
	\tau=\gyr\left[a\c f(b), \gyr[a,f(b)]\sigma_b(f)g(c)\right]\, \gyr[a,f(b)]\gyr[\sigma_b(f)(g(c)), f(b)] \atop \sigma_{b\c g(c)}(f)\gyr[b,g(c)]   \sigma_c(g)h.
\end{align*}
Using Proposition \ref{gyro identity_2}, we get
$$\gyr\left[a\c f(b), \gyr[a,f(b)]\sigma_b(f)g(c)\right]\gyr[a,f(b)]\,\gyr[\sigma_b(f)(g(c)), f(b)] =\gyr[a,f(b\c g(c))].$$
Therefore,
	$$\tau= \gyr[a,f(b\c g(c))]\sigma_{b\c g(c)}(f)\gyr[b,g(c)])\sigma_c(g))h$$
Hence the operation is associative.

\noindent\textbf{Existence of Identity:} 	The pair $(e, I)$ is the identity element.
	
\noindent\textbf{Existence of Inverse:} Let $(a,f)\in G\times\sym(G^*)$. Then

	\begin{eqnarray*}
		(a,f)\cdot (f^{-1}(a^{-1}), \sigma_{a^{-1}}(f^{-1}))&=&(a\c ff^{-1}(a^{-1}), \gyr[a,ff^{-1}(a^{-1})]\sigma_{f^{-1}(a^{-1})}(f)\sigma_{a^{-1}}(f^{-1}))\\
		&=&(e, \sigma_{f^{-1}(a^{-1})}(f)\sigma_{a^{-1}}(f^{-1})) \textnormal{ since $\gyr[a,a^{-1}]=I$}\\
		&=&(e,I) \textnormal{(by Statement (6) of Proposition \ref{identities})}
	\end{eqnarray*}
	Similarly, 	$ (f^{-1}(a^{-1}), \sigma_{a^{-1}}(f^{-1}))\cdot (a,f)= (c,g)$, where
	$c=f^{-1}(a^{-1})\c \sigma_{a^{-1}}(f^{-1})(a)$ and $g=\gyr[f^{-1}(a^{-1}), \sigma_{a^{-1}}(f^{-1})(a)]\sigma_a(\sigma_{a^{-1}}(f^{-1}))(f)=\gyr[f^{-1}(a^{-1}), \sigma_{a^{-1}}(f^{-1})(a)]$. Then, by Proposition \ref{identities}(3), $c=e$  and by Proposition \ref{identities}(7) $g=I.$ Hence, The element $(f^{-1}(a^{-1}), \sigma_{a^{-1}}(f^{-1}))$ is the inverse of $(a,f).$
	
This shows that $(G\times\sym(G^*),\,\cdot\,)$ is a group.
	
	
\end{proof}
\noindent \textbf{Grothendieck group completion}
We identify $\sym(G^*)$ in the group $G\times \sym(G^*)$ with the subgroup $\{e\}\times \sym(G^*)$. Also, we identify the set $G$ in the group $G\times \sym(G^*)$ with $G\times\{I\}$.  Let  $R(G)$ denote  the subgroup of $\sym(G^*)$ generated by 
$\{\gyr[a,b]\mid a,b\in G\}.$ Then the subgroup of $G\times \sym(G^*)$ corresponding to $R(G)$ is  $\{e\}\times R(G)$ and the subgroup $\langle G\rangle$ generated by $G$ in $G\times \sym(G^*)$ is $ G\times R(G)$ which is denoted by  $G R(G)$. 

Consider the group $M(G) = \frac{G R(G)}{[R(G)]}$, where $[R(G)]$ denote the smallest normal subgroup of $ G\times R(G)$ containing $R(G)$. We have a natural quotient homomorphism $\nu : G \longrightarrow M(G)$ given by $\nu (a) = a[R(G)]$. 
 The pair $(M(G), \nu)$ is called the Grothendieck  group completion (GGC) of $G$. By the definitions of $\nu$ and $M(G)$, we observe the following result.

\begin{theorem}
	Let $G$ be a gyrogroup and $M(G)$ be the GGC of $G$. Let $K$ be any group and $\eta: G \to K$ be a gyrogroup homomorphism. Then there exists a unique homomorphism $\alpha: M(G)\to K$ such that $\alpha\circ \nu=\eta,$ that is, the following diagram commutes.
		\[
	\begin{tikzcd}
		G\arrow{r}{\nu} \arrow[swap]{d}{\eta} & M(G) \arrow{dl}{\exists \alpha} \\
		K&  
	\end{tikzcd}
	\]
\end{theorem}

\begin{theorem}\label{p-gyro}
	Let $G$ be a gyrogroup of order $p$,where $p$ is a prime. Then  $M(G)$ is non-trivial.
\end{theorem}
\begin{proof} Since $|G|=p.$ Then the order of the group $GR(G)$ is $pk$, where $k$ is the order of $R(G)$ and $k\mid (p-1)!$. It is clear that $\text{gcd}(p,k)=1.$ Now we claim that $[R(G)]\ne GR(G).$ If equality holds, then $[R(G)]$ has an element of order $p$. But any element of $[R(G)]$ is of the form $ghg^{-1},$ where $g\in GR(G)$ and $h\in R(G).$ Note that $|ghg^{-1}|\mid |h|$ and $|h|\mid k$. Therefore, $[R(G)]$ has no element of order $p$ and hence, $[R(G)]\ne GR(G).$  Therefore, $M(G)$ is non-trivial.
\end{proof}

\begin{proposition}\label{P1}
Let $G_1$ and $G_2$ be two gyrogroups and $\phi$ be a gyrogroup homomorphism from $G_1$ to $G_2$. Then there exists a group homomorphim $\hat{\phi}$ from the Grothendieck group completion $M(G_1)$ of $G_1$ to the Grothendieck group completion $M(G_2)$ of $G_2.$
	\end{proposition}
	
	\begin{proof}
		Suppose $\phi:G_1\longrightarrow G_2$ is a gyrogroup homomorphism and $\nu_i:G_i\longrightarrow M(G_i),\ a\mapsto a[R(G_i)], \ i=1,2$ be the natural quotient maps. Then we have a map $\nu_2 \circ \phi$ from the $G_1$ to $M(G_2)$. Therefore by universal property of the Grothendieck group completion of gyrogroups there exists a group homomorphism $\hat{\phi}:M(G_1)\longrightarrow M(G_2)$ such that $\hat{\phi} \circ \nu_1=\nu_2 \circ \phi.$ 
		\[
		\begin{tikzcd}
		G_1 \arrow{r}{\phi} \arrow[swap]{d}{\nu_1} & G_2 \arrow{d}{\nu_2} \\
		M(G_1)  \arrow{r}{\hat{\phi}} & M(G_2)
		\end{tikzcd}
		\]
	
		\end{proof}

Consider the category $GR$ of gyrogroups and the category $GP$ of groups. Define the functor $$F_1 : GR\longrightarrow GP \ \text{by} \ F_1(G) = M(G).$$ If $G_1,\ G_2$ are two members of the object class of $GR$ and $\phi\in Mor(G_1,G_2)$ then the morphism between $F_1(G_1)$ and  $F_1(G_2)$ is given by
$F_1(\phi)=\hat{\phi}$, where $\hat{\phi}:M(G_1)\longrightarrow M(G_2)$ is a group homomorphism defined in Proposition \ref{P1}.

Also, we have a forgetful functor $$F_2 : GP \longrightarrow GR \ \text{by}\  F_2(\mathcal{G}) = \mathcal{G}_F,$$ where $\mathcal{G}_F$ is treated as gyrogroup.
\begin{theorem}
$F_1$ is a left adjoint functor to $F_2$.
\end{theorem}
\begin{proof}
Let $G\in GR$ and $\mathcal{G} \in GP$. Define $$\Phi_{G,\mathcal{G}}: Hom(F_1(G), \mathcal{G}) \longrightarrow Hom(G, F_2(\mathcal{G}))$$ $$\text{by}\ \Phi_{G,\mathcal{G}} (\eta) = \eta \circ \nu,\ \forall \ \eta \in Hom(F_1(G), \mathcal{G}).$$ Then we claim that the map $\Phi_{G,\mathcal{G}}$ is bijective.
Let $\eta, \chi \in Hom(F_1(G), \mathcal{G})$ such that $$\Phi_{G,\mathcal{G}}(\eta)=\Phi_{G,\mathcal{G}}(\chi).$$ Then,
$\eta \circ \nu=\chi \circ \nu$, by right cancellation law $\eta=\chi.$ Hence, $\Phi_{G,\mathcal{G}}$ is one-one. Now if $\phi\in Hom(G, F_2(\mathcal{G}))$, then by universal property of completion of Grothendieck groups, there exists $\hat{\phi}\in Hom(F_1(G), \mathcal{G})$ such that $\Phi_{G,\mathcal{G}}(\hat{\phi})=\phi $.

Now, let $G,G' \in GR$, $f \in Hom(G',G)$ and $\mathcal{G}_1,\ \mathcal{G}_2 \in GP$ and $g \in Hom(\mathcal{G}_1, \mathcal{G}_2)$.  Then we have two maps $$\Psi_{G,G'}:Hom(F_1(G),\mathcal{G}_1)\longrightarrow Hom(F_1(G'),\mathcal{G}_2)$$ and $$\Psi_{\mathcal{G}_1,\mathcal{G}_2}:Hom(G,F_2(\mathcal{G}_1))\longrightarrow Hom(G',F_2(\mathcal{G}_2))$$ defined by
\begin{center}
$\Psi_{G,G'}(\eta)=g\circ \eta \circ F_1(f)$ \\
$\Psi_{\mathcal{G}_1,\mathcal{G}_2}(\tau)=g\circ \tau \circ f$,
\end{center}

\textbf{Claim:} $\Phi_{G',\mathcal{G}_2}\circ \Psi_{G,G'}=\Psi_{\mathcal{G}_1,\mathcal{G}_2}\circ \Phi_{G,\mathcal{G}_1}$, i.e. the following diagram is commutative:
\[\begin{tikzcd}
Hom(F_1(G),\mathcal{G}_1) \arrow{r}{\Phi_{G,\mathcal{G}_1}} \arrow[swap]{d}{\Psi_{G,G'}} & Hom(G,F_2(\mathcal{G}_1)) \arrow{d}{\Psi_{\mathcal{G}_1,\mathcal{G}_2}} \\
Hom(F_1(G'),\mathcal{G}_2) \arrow{r}{\Phi_{G',\mathcal{G}_2}} & Hom(G',F_2(\mathcal{G}_2)),
\end{tikzcd}
\] 
where $$\Phi_{G',\mathcal{G}_2}: Hom(F_1(G'), \mathcal{G}_2) \longrightarrow Hom(G', F_2(\mathcal{G}_2))$$ defined by $\Phi_{G',\mathcal{G}_2} (\eta') = \eta' \circ \nu'$, $\eta' \in Hom(F_1(G'),\mathcal{G}_2)$ and $\nu' : G' \longrightarrow M(G')$ is the quotient map.

For this, let $\eta \in Hom(F_1(G),\mathcal{G}_1)$. Then 
$$\Phi_{G',\mathcal{G}_2}(\Psi_{G,G'}(\eta))=\Phi_{G',\mathcal{G}_2}(g \circ \eta \circ F_1(f))$$
$$= g\circ \eta \circ F_1(f) \circ \nu' = g\circ \eta \circ \nu \circ f \ (\text{since}\ F_1(f) \circ \nu' = \nu \circ f)$$
  $$\Longrightarrow \Phi_{G',\mathcal{G}_2}(\Psi_{G,G'}(\eta)) = \Psi_{\mathcal{G}_1,\mathcal{G}_2}(\eta \circ \nu)=\Psi_{\mathcal{G}_1,\mathcal{G}_2}(\Phi_{G,\mathcal{G}_1} (\eta) ).$$
  Thus $\Phi_{G',\mathcal{G}_2}\circ \Psi_{G,G'}=\Psi_{\mathcal{G}_1,\mathcal{G}_2}\circ \Phi_{G,\mathcal{G}_1}$.
This proves that $F_1$ is left adjoint to $F_2$.
\end{proof}

\section{Gyrogroup action through its GGC}
In this section, we show that there is a one to one correspondence between the category of gyrogroup actions and the category of group actions.
\begin{proposition}\label{L-subgyrogroup}
Let $G_1$ be a gyrogroup, $G_2$ be a group and $f: G_1 \longrightarrow G_2$ be a gyrogroup homomorphism. Then $f^{-1}(H)$ is a L-subgyrogroup of $G_1$, where $H$ is a subgroup of $G_2$.
\end{proposition}
\begin{proof}
It is easy to see that $f^{-1}(H)$ is a subgyrogroup of $G_1$. Now we show that $\gyr[a, b](f^{-1}(H))\subseteq f^{-1}(H)$.

Let $c \in f^{-1}(H)$. Then $f(c) \in H$. Now,
\begin{align*}
f(\gyr[a, b](c)) &= f(-(a+b) + (a+(b+c)))\\
&= f(-(a+b)) + f(a+(b+c))\\
&= -f(a+b) + f(a) + f(b+c)\\
&= -f(b) - f(a) +f(a) + f(b) +f(c)\\
& = f(c) ~(\text{by associativity in the group})
\end{align*}
Thus, $f(\gyr[a, b](c)) \in H \Rightarrow \gyr[a, b](c) \in f^{-1}(H)$.  By Proposition \ref{l-subgyro}, $f^{-1}(H)$ is a L-subgyrogroup of $G_1$.
\end{proof}
\begin{remark}
Let $G$ be a gyrogroup, $M(G)$ be its Grothendieck group completion and $\nu: G \longrightarrow M(G)$ be the natural quotient homomorphism. By Proposition \ref{L-subgyrogroup}, $\nu^{-1}(H)$ is a L-subgyrogroup of $G$, where $H$ is a subgroup of $M(G)$.
\end{remark}

\begin{theorem}\label{equivalent action} Let $X$ be any non-empty set, $G$ be a gyrogroup and $M(G)$ be its Grothendieck group completion. Then there is a one-one correspondence between the set of all gyrogroup actions of $G$ on $X$ and  the set of all  group actions of $M(G)$ on $X.$
\end{theorem}
\begin{proof}
	Let $\cdot : G \times X \longrightarrow X$ be an action of $G$ on a  set $X$. Then we have a gyrogroup homomorphism $\eta : G \longrightarrow \sym(X)$. Now by universal property of $M(G)$, we have a unique group homomorphism $\alpha : M(G)  \longrightarrow \sym(X)$ such that $\alpha\circ \nu = \eta$. That is we have an action of $M(G)$ on $X$. 
	
	Conversely, if we take an action of $M(G)$ on $X$ that is a group homomorphism $\alpha : M(G)  \longrightarrow \sym(X)$, then we have a gyrogroup homomorphism $\alpha \circ\nu : G \longrightarrow \sym(X)$. That means we have an action of $G$ on the set $X$.
\end{proof}

In the following result, the notations are the same as in Theorem 3.3.
\begin{theorem}\label{equivalence} For $x \in X$, we have the following:
\begin{enumerate}
\item $\orb_G (x) = \orb_{M(G)} (x)$.
\item $\textnormal{Fix}_G (X) = \textnormal{Fix}_{M(G)} (X)$.
\item  Let $Y\subseteq X.$ Then $GY= Y\Leftrightarrow M(G)Y=Y.$ 
\item $\stb_G (x) = \nu^{-1}(\stb_{M(G)} (x))$. Moreover, $\nu(\stb_G (x)) = \stb_{M(G)} (x)$.
\item There is a bijective map between $G/\stb_G (x)$ and $M(G)/\stb_{M(G)} (x)$.
\item The gyrogroup $G$ acts transitively on $X$ if and only if $M(G)$ acts transitively on $X.$
\item The sets $X$ and $Z$ are equivalent under the action of $G$ if and only if  $X$ and $Z$ are equivalent under the action of $M(G)$.

\end{enumerate}
\end{theorem}
\begin{proof}
\begin{enumerate}
\item Let $y \in \orb_G (x)$. Then $y = \eta(a)(x)$, for some $a\in G$. This implies that $y = \alpha\nu(a)(x)= \alpha(\nu(a))(x)$. This shows that $\orb_G (x) \subseteq \orb_{M(G)} (x)$. Conversely, let $y \in \orb_{M(G)} (x)$. Then $y = \alpha(b)(x)$, for some $b\in M(G)$. Since $\nu$ is surjective, $\nu(a)= b$ for some $a\in G$. Thus $y = \alpha(\nu(a))(x) \Rightarrow y = \eta(a)(x)$. This shows that $\orb_G (x) \subseteq \orb_{M(G)} (x)$. This shows that $\orb_{M(G)} (x) \subseteq \orb_G (x)$. Hence, $\orb_G (x) = \orb_{M(G)} (x)$.

\item Let $y \in \textnormal{Fix}_G (X)$. Then $y = \eta(a)(y)$, for all $a\in G$. 
Let $b\in M(G).$ Then $\alpha(b)(y)=\eta(c)(y)=y$ for some $c\in G$ such that $\nu(c)=b$. Therefore, $y\in \textnormal{Fix}_{M(G)}(X).$ Conversely, let $y\in \textnormal{Fix}_{M(G)}(X)$. Then $y=\alpha(b)(y)$ for all $b\in M(G)$. Let $a\in G$. Then $\eta(a)(y)=\alpha(\nu(a))(y)=y$ as $\nu(a)\in M(G).$ Thus, $\textnormal{Fix}_G (X) = \textnormal{Fix}_{M(G)} (X)$.

\item \begin{align*}
	\{\eta(a)(y)\mid a\in G\}= Y&\Leftrightarrow 	\{\alpha(\nu(a))(y)\mid a\in G\}= Y\\
&\Leftrightarrow 	\{\alpha(b)(y)\mid b\in M(G)\}= Y, \textnormal {since $\nu$ is surjective}.
\end{align*}

\item \begin{align*}
a \in \stb_G (x) &\Leftrightarrow \eta(a)(x) = x\\
a \in \stb_G (x) &\Leftrightarrow \alpha(\nu(a))(x) = x\\
a \in \stb_G (x) &\Leftrightarrow \nu(a) \in \stb_{M(G)} (x)\\
a \in \stb_G (x) &\Leftrightarrow a \in \nu^{-1}(\stb_{M(G)} (x))
\end{align*}
Since $\nu$ is surjective, $\nu(\nu^{-1}(\stb_{M(G)} (x))) = \stb_{M(G)}(x)$. This implies that $\nu(\stb_G (x)) = \stb_{M(G)} (x)$.
\item Define $\bar{\nu}: G/\stb_G (x) \longrightarrow M(G)/\stb_{M(G)} (x)$ by $\bar{\nu}(g+\stb_G (x)) = \nu(g) + \stb_{M(G)} (x)$. By using (2), it is easy to see that $\bar{\nu}$ is a bijective map. 

\item Let $x,y\in X.$ 
\begin{align*}
G \textnormal{ acts transitively on } X&\Leftrightarrow 
y=\eta(a)(x) \textnormal{ for some $a\in G$}\\
&\Leftrightarrow y=\alpha(\nu (a))(x) \textnormal{ for some $a\in G$}\\
&\Leftrightarrow y=\alpha(b)(x) \textnormal{ for some $b\in M(G)$}\\
&\Leftrightarrow M(G) \textnormal{ acts transitively on } X
\end{align*}
\item  Let $\phi:X\to Z$ be a bijective map.  Then 
\begin{align*}
	\phi(\eta(a)(x))=\eta(a)(\phi(x))&\Leftrightarrow 	\phi(\alpha(\nu(a))(x))=\alpha(\nu(a))(\phi(x)) \textnormal{ $\forall\, a\in G$ and $x\in X$}\\
	&\Leftrightarrow \phi(\alpha(b)(x))=\alpha(b)(\phi(x)) \textnormal{  $\forall\,b\in M(G)$ and $x\in X$}
\end{align*} 
\end{enumerate}
\end{proof}
\begin{remark}
		\begin{enumerate}
		\item By orbit-stabilizer theorem of group action, we have a bijective correspondence between  $\orb_{M(G)} (x)$ and  $M(G)/\stb_{M(G)} (x)$. By using Theorem \ref{equivalence} (1) and (5), we get a bijective correspondence between  $\orb_G (x)$ and  $G/\stb_G (x)$ which is the orbit-stabilizer theorem of gyrogroup action \cite[Theorem 3.9]{suksumran_gyroaction}.
		
		Suppose $G$ is a finite gyrogroup. Then $M(G)$ is a finite group. Therefore, 
		$|M(G)| = |\orb_{M(G)} (x)| |\stb_{M(G)} (x)| \Rightarrow |G| = |\orb_{G} (x)| |\stb_{G} (x)|$.
		
		\item Let $x_1, x_2,\dots, x_k$ be representatives for the distinct nonsingleton orbits in $X$ by action of $M(G)$. Then by orbit-decomposition theorem for group action, we have
		$$|X|=|\textnormal{Fix}_{M(G)}(X)|+\sum_{i=1}^k[G:\textnormal{stab}_{M(G)}(x_i)].$$
		By using Theorem \ref{equivalence} (2) and (5), we obtain orbit-decomposition theorem for gyrogroup action \cite[Theorem 3.10]{suksumran_gyroaction}, that is,
			$$|X|=|\textnormal{Fix}_{G}(X)|+\sum_{i=1}^k[G:\textnormal{stab}_{G}(x_i)].$$
		
	\item 	By Burnside lemma for groups, the number of distinct orbits in $X$ by action of $M(G)$ is
\begin{align*}
	 \frac{1}{|M(G)|}\sum_{b\in M(G)}|\textnormal{fix}_{M(G)} (b)|
\end{align*}
	It is easy to see that $$|G|=|\ker (\nu)|\, |M(G)| \textnormal{ and } \sum_{b\in M(G)}|\textnormal{fix}_{M(G)} (b)|=\frac{1}{|\ker(\nu)|}\sum_{a\in G}|\textnormal{fix}_{G} (a)|.$$
	
	\noindent Therefore, the Burnside lemma for gyrogroup action holds true \cite[Theorem 3.11]{suksumran_gyroaction}.  
	
\item The Lagrange theorem  in gyrogroups are the consequences of the Theorem \ref{equivalence}.
		
		\item If gyrogroup action of $G$ on $X$ is faithful, then the corresponding group action of $M(G)$ on $X$ is also faithful. Similarly,  if gyrogroup action of $G$ on $X$ is sharply transitive, then the corresponding group action of $M(G)$ on $X$ is also sharply transitive. In both the cases, converse is true if $\nu$ is injective, that is, $G=M(G).$
	\end{enumerate}

\end{remark}

The following result is done by T. Suksumran and K. Wiboonton in \cite{suksumran_lagrange}. Here, we obtain the result by using GGC.  

\begin{theorem}
	If $G$ is a gyrogroup of prime order $p$, then $G$ is a cyclic group of order $p$ under gyrogroup operation.
\end{theorem}
\begin{proof}
	By Theorem \ref{p-gyro}, $M(G)$ is non-trivial and hence, $1<|M(G)|\le p$.  Note that $G/\ker(\nu)\cong M(G)$ and $|\ker(\nu)|$ is 1 or $p.$ If $|\ker(\nu)|=p$, the $M(G)$ is trivial, which is not possible. Therefore, $\ker(\nu)$ is trivial so that $G\cong M(G).$ Hence, $G$ s a cyclic group of order $p$ under gyrogroup operation.
\end{proof}
\section{Gyrogroup representation through its GGC}
Now, we discuss the representation of a gyrogroup. Let $G$ be a gyrogroup, $M(G)$ be its GCC and $V$ be a vector space over an algebraically closed field $k$.

\begin{theorem}\label{equivalent representation} Let $X$ be any non-empty set, $G$ be a gyrogroup and $M(G)$ be its Grothendieck group completion. Then there is a one-one correspondence between the set of all gyrogroup representation $G \longrightarrow GL(V)$ and  the set of all  group representation of $M(G)\longrightarrow GL(V)$.
\end{theorem} 
\begin{proof}
Similar to the proof of Theorem \ref{equivalent action}.
\end{proof}

In \cite{suksumran_repre}, the author defined the vector space $L^{\gyr}(G) = \{f: G \longrightarrow k \mid f\circ (L_a \circ \gyr[x, y] \circ L_a^{-1}) = f, \forall ~ a, x, y \in G\}$ over $k$, where $k$ is an algebraically closed field. He showed that $G$ acts linearly on $L^{\gyr}(G)$, where action $*$ is defined by $(g*f)(x) = f(g^{-1}\circ x)$ for all $g\in G$ and $f \in L^{\gyr}(G)$. Thus we have corresponding gyrogroup representation $\rho : G \longrightarrow GL(L^{\gyr}(G))$ defined by $\rho(g)(f) = g*f$.
\begin{theorem}
There is an vector space isomorphism between  $L^{\gyr}(G)$ and $k(M(G))$, where $k(M(G))$ is the group algebra of $M(G)$ over $k$. Moreover, if $G$ is a finite gyrogroup, then $\text{dim}(L^{\gyr}(G)) = |M(G)|$.
\end{theorem}
\begin{proof}
Define the map $\phi: k(M(G))\longrightarrow L^{\gyr}(G)$ as $\phi (F) = F\circ\nu$, where $F \in k(M(G))$ and $\nu : G \longrightarrow M(G)$ is the natural quotient gyrogroup homomorphism. It is easy to verify that $\phi$ is an injective linear map.

Now, we will show that $\phi$ is surjective. Let $f \in L^{\gyr}(G)$. Then, for  all $ a, x, y, x', y' \in G$
\begin{align*}
f\circ (L_a \circ \gyr[x, y] \circ L_a^{-1}) &= f\\
\Rightarrow f\circ (L_a \circ \gyr[x, y] \circ\gyr[x', y'] \circ L_a^{-1}) &= f\circ (L_a \circ \gyr[x, y] \circ L_a^{-1} \circ L_a \circ\gyr[x', y'] \circ L_a^{-1})\\
&= (f\circ (L_a \circ \gyr[x, y] \circ L_a^{-1}) \circ (L_a \circ\gyr[x', y'] \circ L_a^{-1}))
\\
&= f \circ (L_a \circ\gyr[x', y'] \circ L_a^{-1}))
\\&= f
\end{align*}
By the similar way, we can show that $f \circ(L_a \circ K \circ L_a^{-1})) = f $, for all $K \in [R(G)]$. In other words, $(L_a \circ \gyr[x, y] \circ L_a^{-1})(z) \in \ker(\nu)$, for all $ a, x, y, z \in G$

Define $F: M(G) \longrightarrow \mathbb{C}$ as $F(g[R(G)]) = f(g)$. By the above argument, clearly, $F$ is well defined. Hence, $\phi (F) = F\circ\nu = f$. Hence, $\phi$ is surjective.

Further, if $G$ is a finite gyrogroup, then $M(G)$ is a finite group and $\text{dim}(k(M(G))) = |M(G)| = \text{dim}(L^{\gyr}(G))$.
\end{proof}

\begin{remark}
We can easily observe the following:
\begin{enumerate}
\item If $\text{Ch}(k) \nmid |G|$, then $\text{Ch}(k) \nmid |M(G)|$.
\item Let $\rho : G \longrightarrow GL(V)$ be a gyrogroup representation of $G$. Then $\rho = \rho'\circ \nu$, for a unique group representation $\rho' : M(G) \longrightarrow GL(V)$ of $M(G)$.
\item $W$ is a sub representation of $G$ if only if $W$ is a sub representation of $M(G)$. Also, every irreducible gyrogroup representation of $G$ on $V$ gives an irreducible group representation of $M(G)$ on $V$ and vice versa.

\item Suppose $\text{Ch}(k) \nmid |G|$ or $\text{Ch}(k) = 0$. Let $\rho : G \longrightarrow GL(V)$ be a gyrogroup representation of $G$. Then $\rho = \rho'\circ \nu$, for a unique group representation $\rho' : M(G) \longrightarrow GL(V)$ of $M(G)$. Now, by Maschke's theorem for group, we have $\rho'= \rho_1'\oplus \cdots \oplus \rho_n'$. Therefore, $\rho= \rho_1\oplus \cdots \oplus \rho_n$, where $\rho_i = \rho_i'\circ \nu$, for $1\leq i \leq n$. This is the Maschke's theorem for gyrogroup.

\item Consider the action $*$ of $M(G)$ on $k(M(G))$ defined as $((g[R(G)])*F)(x[R(G)]) = F((g^{-1} \circ x)[R(G)])$. Thus $(\nu(g) * F)(\nu(x)) = (F\circ \nu)(g^{-1} \circ x)$.

Now, let $f\in L^{\gyr}(G)$. Then $f = F\circ \nu$ for some $F \in k(M(G))$. Thus, from the motivation of $*$ of $M(G)$ on $k(M(G))$, we have action $*$ of $G$ on $L^{\gyr}(G)$ defined by $(g*f)(x) = f(g^{-1} \circ x)$ which is the action introduced in \cite{suksumran_repre}. 
\item The restriction of representation (action) of $GR(G)$ induced by the trivial representation (action) of $[R(G)]$ on $G$ is the same as representation  (action) defined in  \cite{suksumran_gyroaction} {(\cite{suksumran_repre})}. Also, from every representation (action) of $[R(G)]$, we have a  representation (action) of $G$.
\item Let $H$ be a L-subgyrogroup of $G$. Then $\nu(H)$ be the subgroup of $M(G)$. By using the concept of induced representation of $M(G)$ by $\nu(H)$, we have an analogous  concept of induced representation of $G$ by $H$. In the similar manner of group representation, we can deduce the Frobenius reciprocity and Mackey’s Irreducibility Criterion for gyrogroup $G$. 
\end{enumerate}

\end{remark}

\section{Right Gyrogroup Action}
In the previous section, we have seen that a gyrogroup action is nothing but a group action by the GGC of the gyrogroup. In this section, we deal with right gyrogroups  and study right gyrogroup action and its properties. The right gyrogroups  were first studied by Tuval Foguel and Ungar \cite{foguel_involutory} as left gyrogroups. For details, we refer to \cite{foguel_involutory,rjl_topological rgyrogroup}.
\begin{definition}\cite{rjl_topological rgyrogroup}
	A groupoid $(G,\circ)$ is called a right gyrogroup if 
	\begin{enumerate}
		\item there is an element $e$ in $G$ such that $a\circ e=a,$ for all $a\in G$ (right identity);
		\item for every $a \in G$, there exists an element $a'\in G$ such that $a\circ a'=e$ (right inverse);
		\item for all $y,z \in G$, there exists an automorphism $\gyr[y,z]:G\longrightarrow G$ such that $(x\circ y)\circ z=\gyr[y,z](x)\circ (y\circ z),$ for all $x\in G$;
		\item for all $y\in G$, $\gyr[y,y']=I_G,$ where $y'$ is the right inverse of $y.$
	\end{enumerate}
\end{definition}
Now, we give some properties of a right gyrogroup discussed in \cite{rjl_topological rgyrogroup}.
\begin{proposition}
	Let $(G,\circ)$ be a right gyrogroup. Then
\begin{enumerate}
	\item $a'$ is also the left inverse of $a,$ for all $a\in G.$
	\item $e$ is also the left identity of $G.$
	\item $\gyr[e,x]=I_G=\gyr[x,e),$ for all $x\in G.$
	\item $\gyr[x,y]\gyr[x\circ y,z]=\gyr[y,z]\gyr[\gyr[y,z](x),y\circ z]$ for all $x,y,z\in G.$
    \item $\gyr[x,x']=I_G,$ for all $x\in G.$
\end{enumerate}
\end{proposition}
\textbf{Note-} Now onwards, we denote $a'$ by $a^{-1}.$
\begin{proposition}\label{group based action}
	Let $(G,\cdot)$ be a group. Define a binary operation $\circ$ on $G$ by $a\circ b=b^{-1}ab^2$.
	Then,  $(G,\circ)$ is a right gyrogroup with gyro automorphism $\gyr[a,b]:G\longrightarrow G$ defined by $\gyr[a,b](c)= [b^{-1},a]c[a, b^{-1}],$ where $[a,b]$ denotes the commutator of $a,b$ in $G.$ We call $(G,\circ)$ the $G$-based right gyrogroup. 
\end{proposition}

Now, we give a more realistic action of a right gyrogroup. Further, we show that every right gyrogroup action gives a right gyrogroup homomorphism and vice-versa.
\begin{definition}
	Let $(G, \circ)$ be a right gyrogroup  and $X$ be a non empty set. An action of $G$ on $X$ is a map $*:G\times X\longrightarrow X$ such that
	\begin{enumerate}
		\item $1*x=x$
		\item $(b\circ a)*x=a^{-1}*(b*(a*(a*x))),$ for all $a,b\in G$ and $x\in X.$
	\end{enumerate}
\end{definition}

\begin{definition}
	Let $(G,\circ)$ be a right gyrogroup and $(G',\odot)$ be a group. A map $\phi:G\to G'$ is called a right gyrogroup homomorphism if $$\phi(a\circ b)=(\phi(b))^{-1}\odot \phi(a)\odot (\phi(b))^2.$$
\end{definition}
\begin{theorem}
 Let $(G, \circ)$ be a right gyrogroup and $*$ be an action of $G$ on a set $X$. For each $a\in G$, define a map $f_a:X\longrightarrow X,$ by $f_a(x)=a*x$. Then\\ 
(1) $f_a\in \sym(X)$.\\
(2) Define $\phi:G\to \sym(X)$  by $\phi(a)=f_a$. Then the map $\phi$ is a right gyrogroup homomorphism from $G$ to $\sym(X)$ with $\ker \phi=\{a\in G:a·*x=x\,\, \forall \,x\in X\}.$\\
(3) Let $\hat{\phi} :G \to \sym(X)$ be a right gyrogroup homomorphism. The map $\hat{*}$ defined by $a \hat{*} x = \phi(a)(x), \,a\in  G,\, x \in X$
is a right gyrogroup action of $G$ on $X$.
\end{theorem}
\begin{proof} Suppose $f_a(x)=f_a(y)$ for some $x,y\in X$. Then \begin{align*}
		 f_a(x)=f_a(y)
	&	\Rightarrow a*x=a*y\\
	&	\Rightarrow a^{-1}*(a^{-1}*(a*(a*x)))
		=a^{-1}*(a^{-1}*(a*(a*y)))\\
	&	\Rightarrow (a\circ a^{-1})*x=(a\circ a^{-1})*y\\
	&	\Rightarrow 1*x=1*y\Rightarrow x=y
	\end{align*}
Hence, $f_a$ is one one. Now let $x\in X.$ Then 
$f_a(a*(a^{-1}*(a^{-1}*x)))=a*(a*(a^{-1}*(a^{-1}*x)))$
$=(a^{-1}\circ a)*x=x.$
Therefore, $f_a$ is onto as well so that $f_a\in\sym(X).$

For $a\in G,$ $(a^{-1}\circ a)*x=a^{-1}*(a^{-1}*(a*(a*x)))$
$\Rightarrow x=((f_{a^{-1}})^2\circ (f_a)^2)(x)$
$\Rightarrow f_{a^{-1}}=f_a^{-1},$ (since $f_{a^{-1}}$ and $f^{-1}_a$ are bijections).
Let $a,b\in G$, then $\phi(a\circ b)=f_{a\circ b}$ and $f_{a\circ b}(x)=(f_{a^{-1}}\circ f_b\circ f_{a}^2)(x)$ (here $\circ$ is composition of maps). 
Thus, $f_{a\circ b}(x)=((f_a)^{-1}\circ f_b\circ f_a^2)(x)$
$\Rightarrow \phi(a\circ b)=(\phi(a))^{-1}\phi(b)(\phi(a))^2,$ i.e., $\phi$ is a right gyrogroup homomorphism. Then
 \begin{align*}
 \ker 	\phi&=\{a\in G\mid \phi(a)=Id_X\}\\
 &=\{a\in G\mid f_a=Id_X\}\\
 &=\{a\in G\mid a*x=x\,\, \forall x\in X\}.
 \end{align*}
To prove the third statement, note that $e\hat{*}x =\hat{\phi}(e)(x) =Id_X(x) =x$ for all $x \in X$. For  $a, b \in G$ and  $x \in X$, $b^{-1} \hat{*}(a \hat{*}(b\hat{*}(b\hat{*}x))) =(\hat{\phi}(b^{-1})\circ \hat{\phi}(a) \circ \hat{\phi} (b)\circ \hat{\phi}(b))(x) =\hat{\phi}(a \circ b)(x) =(a \circ b) \hat{*}x. $ 
\end{proof}	


\begin{proposition}
Let $G$ be group and $G,\circ$ be $G$-based right gyrogroup.	Then every group action of $G$ on a set $X$ is a right gyrogroup action of $(G,\circ)$ on $X$. 
\end{proposition}
\begin{proof}
	Let $*$ be an action of $G$ on $X$. Then $e*x=x$ and $(a\circ b)*x=(b^{-1}ab^2)*x=b^{-1}*(a*(b*(b*x))$. This shows that $*$ is a right gyrogroup action of $(G,\circ)$ on the set $X$.  
\end{proof}

The following example shows that a right gyrogroup action of $(G,\circ)$ need not be a group action of $G$.
\begin{example}
Consider the non-abelian group $G=\langle a,b|a^3=b^3=e=[a,b]^3=[a,[a,b]]=[b,[a,b]]\rangle$ of order $27$.  Then the $G$-based right gyrogroup $(G,\circ)$ is an abelian group.

Let $*$ be  the left action of the group $(G,\circ)$ on $G$, that is,  $g*x=g\circ x$. Then, it is easy to see that 
\begin{align*}
(g_1 \circ g_2)* x &= (g_1 \circ g_2)\circ x \\&= (g_1 \circ g_2)\circ x 
\\&= g_2^{-1} \circ(g_1 \circ (g_2 \circ (g_2\circ x))) ~(\text{by using commutative property of }\circ)
\\&= g_2^{-1} *(g_1 *(g_2 * (g_2* x)))
\end{align*}
This shows that $*$ is a right gyrogroup action of $(G,\circ)$. We can also conclude it by the following argument:

Since $(G,\circ)$ is a group and $*$ be an action of $(G,\circ)$, by Proposition  \ref{group based action}, $*$ is also a right gyrogroup action of $(G,\circ_1)$, where $(G,\circ_1)$ is $(G,\circ)$-based right gyrogroup. Since $(G,\circ)$ is an abelian group, $g_1 \circ_1 g_2 = g_2^{-1}\circ g_1 \circ g_2^{2} = g_1 \circ g_2$. Hence, $*$ is a right gyrogroup action of $(G,\circ)$.

Now finally we show that $*$ is not a group action of $G$. It is enough to show that
$$(ab^{-1})*a\neq a*(b^{-1}*a).$$
Suppose, 
$(ab^{-1})*a=a*(b^{-1}*a)$

$\Rightarrow a^{-1}(ab^{-1})a^2=a*(a^{-1}b^{-1}a^2)$

$\Rightarrow b^{-1}a^2=(a^{-2}ba)(a)(a^{-1}b^{-1}a^2)(a^{-1}b^{-1}a^2)$

$\Rightarrow b^{-1}a^2=abab^{-1}ab^{-1}a^2$

$\Rightarrow e=a[b,a]a^2$

$\Rightarrow ab=ba,$ which is a contradiction. Therefore, $*$ is not an action of $G$.
\end{example}
Next, we discuss an example of a right gyrogroup for which every right gyrogroup action is trivial.
\begin{example}\cite[Corollary 5.13]{rjl_topological rgyrogroup}
	Let $G$ be a group and $H$ be a subgroup of $G$. Let $S$ be a right transversal of $H$ in $G$. Then $(S, \circ)$ is  a groupoid where $\circ$ is defined by $$x\circ y=Hxy\cap S.$$
	 Now, consider $G=Sym(n),$ $H=Sym(n-1)$ and $S=\{I,(1\ n),(2 \ n), \cdots (n-1 \ n)\}.$ For $1\leq i\ne j\leq n-1$, we have
	 
	 $H(i \ n)(j \ n)=H(i \ n \ j)=H(i \ j)(i \ n)=H(i \ n).$
	 
	 Hence, $(i \ n )\circ (j n)=(i \ n),$ for $i\neq j$.
 
Then $(S, \circ)$ is a  right gyrogroup. In this case, $\gyr[a, b]: S \longrightarrow S$ is given by 

$\gyr[a, b](c)= \begin{cases}
c, ~\text{if}~ c= b\\
b, ~\text{if}~ a= c\\
c, ~\text{if}~ a= b\\
c, ~\text{if}~ a\ne b \ne c
\end{cases}$
	 
	 Let $*$ be an action of $S$ on a non empty set $X$. Then we have a homomorphism $\phi:S\longrightarrow Sym(X)$ by
	 
	 $\phi(x\circ y)=\phi(y)^{-1}\phi(x)\phi(y)^2$
	 
	 $\Rightarrow \phi((i \ n) \circ (j \ n))=\phi((j \ n))^{-1} \phi(i \ n) \phi((j \ n ))^2$
	 
	 $\Rightarrow \phi((i \ n))=\phi((j \ n))\phi((i \ n))\phi((j \ n))^2$
	 
	 $\Rightarrow \phi((i \ n))=\phi((j \ n))\phi((i \ n))$
	 
	 $\Rightarrow \phi((j \ n))=I,$ for all $1\leq j\leq n-1.$
	 
	 Hence $\phi$ is a trivial homomorphism, which implies $*$ is a trivial right gyrogroup action.
\end{example}
 
 Here, we use the fact that a right gyrogroup is a particular case of right loops and adopt certain notions and terminologies from \cite{shukla}   to define an invariant right subgyrogroup. Further, we show that the isotropy subgyrogroup of a right gyrogroup action is an invariant subgyrogroup.
\begin{definition}\cite{shukla}
An equivalence relation $R$ on a right gyrogroup $G$ is called a congruence in $G$ if $R$ is a right subgyrogroup of $G\times G$.
\end{definition}

The following result is a particular case of \cite[Theorem 2.7]{shukla}.
\begin{theorem}\label{congruence}
Let $R$ be a congruence on a right gyrogroup $G$ and $R_e$ the equivalence class determined by the identity $e$. Then 
\begin{enumerate}\item $R_e$ is a right subgyrogroup of $G$. 
\item $\gyr[y,z](x)\in R_e$ and $\gyr^{-1}[y,z](x)\in R_e$ for all $x\in R_e$ and $y, z \in G$.
\item $(y\circ (x \circ z)) \circ (y \circ z)^{-1} \in R_e$ for all $x\in R_e$ and $y, z \in G$.
\item $[\gyr^{-1}[\gyr[x,y](x^{-1}), x \circ y](z)\circ \gyr[x,y](x^{-1})]\circ z^{-1} \in R_e$ for all $x\in R_e$ and $y, z \in G$.
\item $(y, z) \in R$ and $y \circ z^{-1} \in R_e$.
\item $R = \{(a, b)\in G\times G\mid a = x\circ b ~\text{for some}~ x \in R_e\}$
\end{enumerate}

Conversely, let $H \ne \{e\}$ be a sub right gyrogroup of $G$ satisfying $(2) -(4)$, then there is a unique congruence $R$ on $G$ such that $R_e = H$.
\end{theorem}
\begin{definition}
A right subgyrogroup of a right gyrogroup $G$ will be called an invariant  right rightgyrogroup if it satisfies the conditions $(2), (3)$ and $(4)$ of Theorem \ref{congruence}.
\end{definition}

\begin{proposition}
Let $G$ be a right gyrogroup and $*$ be an action of $G$ on a non empty set $X$. Then the set $G_X=\{a \in G \mid a* x = x ~\forall x \in X \}$ is an invariant right subgyrogroup of $G$.
\end{proposition}
\begin{proof}
First, we show that $G_X$ is a right subgyrogroup. Let $a\in G_X$. Then $a^{-1}*x=a^{-1}*(e*(a*(a*x)))=(e\circ a)*x=a*x=x\,\,\forall\, x\in X.$  Let $a,b\in G_X.$ Then $(a\circ b^{-1})*x=b*(a*(b^{-1}*(b^{-1}*x)))=x\,\,\forall \, x\in X.$ Thus, $G_X$ is a right subgyrogroup. 
For $x\in X,$ $a\in G_X$ and $b,c\in G,$
\begin{align*}
	((a\circ b)\circ c)*x&= c^{-1}*((a\circ b)*(c*(c*x)))\\
	&=c^{-1}*(b^{-1}*(a*(b*(b*(c*(c*x)))))\\
	&=c^{-1}*(b^{-1}*(b*(b*(c*(c*x))))\\
	&=c^{-1}*((e\circ b)*(c*(c*x)))\\
	&=c^{-1}*(b*(c*(c*x)))\\
	&=(b\circ c)*x.
\end{align*}
 Therefore, 
 \begin{align*}
 	\gyr[b,c](a)*x&=((a\circ b)\circ c)\circ (boc)^{-1})*x\\
 	&=(b\circ c)*((a\circ b)\circ c))*(b\circ c)^{-1}*(b\circ c)^{-1}*x\\
 	&=(b\circ c)*( b\circ c))*(b\circ c)^{-1}*(b\circ c)^{-1}*x\\
 	&=x.
 \end{align*}
Thus,  $\gyr[b,c](a)\in G_X$ for $a\in G_X$ and $b,c\in G.$ Note that
$\gyr^{-1}[b,c](a)=\gyr[b^{-1},c^{-1}](a)$ and hence, $\gyr^{-1}[b,c](a)*x=x$.

Observe, for $a\in G_X$ and $b\in G$, $(a\circ b)*x=b*x$ and $(a\circ b)^{-1}*x=b^{-1}*x.$ Now suppose $a\in G_X$ and $b,c\in G.$
\begin{align*}
	(b\circ(a\circ c))*x&=(a\circ c)^{-1}*(b*((a\circ c)*((a\circ c)*x)))\\
	&=c^{-1}*(b*((a\circ c)*((a\circ c)*x)))\\
		&=c^{-1}*(b*((c*(c*x)))\\
		&=(b\circ c)*x.
\end{align*}	
Therefore,
\begin{align*}
(	(b\circ(a\circ c))\circ(b\circ c)^{-1})*x&=(b\circ c)*((b\circ(a\circ c))*((b\circ c)^{-1}*((b\circ c)^{-1}*x)))\\
&=(b\circ c)*((b\circ c)*((b\circ c)^{-1}*((b\circ c)^{-1}*x)))\\
&=((b\circ c)\circ (b\circ c)^{-1})*x\\
&=x.
\end{align*}

Observe that $[\gyr^{-1}[\gyr[a,b](a^{-1}), a \circ b](c)\circ \gyr[a,b](a^{-1})]\circ c^{-1}=((c\circ b)\circ(a\circ b)^{-1})\circ c^{-1}.$
Therefore,
\begin{align*}
	(((c\circ b)\circ(a\circ b)^{-1})\circ c^{-1})*x&=c*(((c\circ b)\circ(a\circ b)^{-1})*(c^{-1}*(c^{-1}*x)))\\
	&=c*((b*((c\circ b)*(b^{-1}*( b^{-1}*((c^{-1}*(c^{-1}*x))))\\
	&=c*((c\circ b)\circ b^{-1})*((c^{-1}*(c^{-1}*x))))\\
	&=c*((c*((c^{-1}*(c^{-1}*x))))\\
	&=x
\end{align*}

 Thus, $G_X$ is an invariant right subgyrogroup of $G.$ 
\end{proof}

\begin{definition}
	Let $H$ be an invariant right subgyrogroup of a right gyrogroup $G$. The set $G/H = \{H\circ g \mid  g \in G\}$ is a right gyrogroup under the following operation  
	$$(H\circ a)\circ(H\circ b)=H\circ(a\circ b),$$
	where $H\circ a, H\circ b\in G/H$. The right gyrogroup $(G/H,\circ)$ is called the quotient right gyrogroup with gyro automorphism $\gyr[(H\circ a),(H\circ b)](H\circ c)=H\circ \gyr[a,b](c).$ 
\end{definition}

\begin{proposition} Let $G$ be a right gyrogroup and $*$ be an action of $G$ on a non empty set $X$. Let $G_X=\{a \in G \mid a* x = x ~\forall x \in X \}$. Then
	$G/G_X$ is isomorphic to the right subgyrogroup $\phi(G)$ of $\sym(X)$,
	where $\phi$ is the corresponding representation.
\end{proposition}

\begin{remark}
	Let $G$ be a right gyrogroup and $*$ be an action of $G$ on a non empty set $X$. The stabilizer of $x$ in G, denoted by $\textnormal{stab } x$, is defined as 
	$\textnormal{stab } x= \{a \in  G: a * x = x\}$. Then $\textnormal{stab }$ is a right subgyrogroup of $G.$ Define a relation $\sim$ on $X$ by the condition
	$x \sim y \Leftrightarrow y = a * x$ for some $a \in G$. Then
	the relation $\sim$ is reflexive, symmetric but not transitive. Then the orbit of $x$ can be defined as the equivalence class of $x$ determined by the transitive closure of the relation $\sim.$\\
\end{remark}

\noindent{\bf Conclusion \& Problems:} In this article, we have shown that GGC of a gyrogroup works as an important tool to study gyrogroup action. We propose that many important results on gyrogroups can be obtained using GGC. One of the problems can be stated as follows:
Classify all those gyrogroups for which their GGC is trivial.\\

\noindent{\bf Acknowledgment}\\

\noindent We are extremely thankful to Prof. Ramji Lal for his continuous support, discussion and encouragement. The first and the second named authors thank IIIT Allahabad and Ministry of Education, Government of India for providing institute fellowship. The third named author is thankful to IIIT Allahabad for providing SEED grant.


\end{document}